\documentclass[11pt,reqno]{amsart}

\usepackage[margin=25mm,a4paper]{geometry}
\usepackage{mathtools}
\usepackage{dsfont}
\usepackage[shortlabels]{enumitem}
\usepackage{amsaddr}

\newtheorem{theorem}{Theorem}[section]

\newtheorem{corollary}[theorem]{Corollary}
\newtheorem{proposition}[theorem]{Proposition}
\newtheorem{definition}[theorem]{Definition}
\theoremstyle{remark}
\newtheorem{remark}[theorem]{Remark}

\numberwithin{equation}{section}

\newcommand{\abs}[1]{\left\lvert#1\right\rvert}
\newcommand{\set}[1]{\left\{#1\right\}}
\newcommand{\pr}{\mathbf P}
\newcommand{\Qpr}{\mathbf Q}
\newcommand{\ex}{\mathbf E}
\newcommand{\ind}{\mathds 1}
\newcommand{\FF}{\mathcal F}
\newcommand{\real}{\mathbb R}

\allowdisplaybreaks

\begin{document}
\title[CIR and squared Bessel processes]{Driven by Brownian motion Cox--Ingersoll--Ross and squared Bessel processes: interaction and phase transition}
\author{Yuliya Mishura$^1$, Kostiantyn Ralchenko$^{1,2}$, Svitlana Kushnirenko$^1$}
\address{$^1$Taras Shevchenko National University of Kyiv\\
$^2$University of Vaasa}
\begin{abstract}
This paper studies two related stochastic processes driven by Brownian motion: the Cox--Ingersoll--Ross (CIR) process and the Bessel process. We investigate their shared and distinct properties, focusing on time-asymptotic growth rates, distance between the processes in integral norms, and parameter estimation. The squared Bessel process is shown to be a phase transition of the CIR process and can be approximated by a sequence of CIR processes. Differences in stochastic stability are also highlighted, with the Bessel process displaying instability, while the CIR process remains ergodic and stable.
\end{abstract}
\maketitle

\section{Introduction}
\subsection{Some historical information} Probability theory (stochastics), which deals with the description and analysis of stochastic objects, is connected with the most diverse phenomena, manifestations of the surrounding reality and  with technical phenomena. Stochastics creates  theoretical basis of applied applications, that is,  adequate models of existing phenomena. At the same time, the created model should allow a simple but flexible image with the help of certain mathematical formulas. Once a good model is created, further development proceeds in two directions: first, the theory of related and more general models itself begins to develop, and second, the model begins to cover more and more applications, since approximately the same regularities are manifested   in physics,   financial mathematics,  biology,   economics,  climatology, and even in some social sciences. Therefore, the same model or any of its variants can be used in a wide variety of areas. Furthermore, as it is well-known, Brownian motion (Wiener process) is the best possibility to involve randomness into the model. 

In 1827, the botanist Robert Brown (1773--1858) first observed the phenomenon, which was later called Brownian motion. A brief description of the observed movement is as follows: imagine a laboratory dish, something like a cup, in which there is a liquid, and in the liquid is poured flower pollen. So, Brown observed the movement of pollen particles in this cup. The phenomenon so impressed him that he described it in detail. Namely, according to observations, the specified movement was chaotic, the trajectories of the particles constantly changed direction, were broken, in fact, at no moment in time did any particle have a fixed direction of movement. Robert Brown could not create a certain simple, non-random description of this system, but only noted that the movement of pollen in a liquid is described by a very chaotic process. From a modern point of view, he actually observed a random process that took on a value of a small part of the plane (the surface of the cup) and that changed in time non-deterministically. At that time, there were no explanations for this phenomenon. They appeared much later.

The next step in the development of the theory of Brownian motion was made by outstanding physicists. Namely, in 1905--1906, the famous works of Albert Einstein (1879--1955) and Marian Smoluchowski (1872--1917), \cite{eins} and \cite{smol1}, were published independently of each other; then their original German papers on the specified topic, in which the phenomenon of Brownian motion was explained, were combined into a collection \cite{einsmol}. English translations of Einstein's papers were published in \cite{eins-eng}.
Einstein and Smoluchowski explained, in particular, the random movement of flower pollen in a liquid by such a phenomenon as the thermal chaotic movement of atoms and molecules. 

According to this theory, liquid or gas molecules are in constant thermal motion, and the impulses of different molecules are unequal in magnitude and direction. If the surface of a particle placed in such an environment is small, as is the case for a Brownian particle, in particular, a particle of flower pollen, then the shocks felt by the particle from the surrounding molecules will not be precisely compensated. Therefore, as a result of ``bombardment'' by molecules, the Brownian particle begins to move randomly, changing the magnitude and direction of its velocity approximately $10^{14}$ times per second. This is how a seemingly purely biological phenomenon was explained from a physical point of view. 

It also turned out that the random process $B_t$ is needed to describe this thermal motion and to create this kinetic theory, i.e.\ Brownian motion is an adequate model for this kinetic theory. Einstein and Smoluchowski not only described this random process from the point of view of its growth and from the point of view of its dependence on cases, but they found the so-called distribution of process values  and formulated a partial differential equation that is satisfied by the   transition density of the distribution of Brownian motion. It turned out that Brownian motion has a Markov property. This was written using so-called transition probabilities, or rather using partial derivative equations for densities, that is, derivatives of transition probabilities. These equations showed that Brownian motion has a Markov property: independence of the future from the past with a fixed current. This property significantly helps to simplify at least the formulas related to transition probabilities. 

Thus, the theory of Brownian motion found a serious reinforcement from physicists and became ``overgrown'' with interesting properties. However, the real world cannot always be described by a linear model, it is much more complex. Therefore, integration over Brownian motion was constructed first with non-random (N. Wiener), and then with random (K. It\^{o}) integrands, and the theory of stochastic differential equations was developed. 

\subsection{Overview of the results presented in the paper}
In the present paper we shall consider two closely related stochastic processes, namely, Cox--Ingersoll--Ross and Bessel process, both of them being strictly positive solutions of the respective stochastic differential equations. Strictly positive values   make  them convenient to model real processes in physics, biology, economics. In finances they are used  to forecast interest rates and in bond pricing models, see e.g.\ \cite{Brigo,DiFrancesco22,Maghsoodi96,Orlando19}. Similar models are used to simulate changes in the membrane voltage of a neuron  \cite{neuron}.
In our research we combine the methods of stochastic analysis and methods based on the explicit formulas for probability distributions of CIR and Bessel processes. 

In a certain sense, the square of the Bessel process can be considered the result of a phase transition in the Cox--Ingersoll--Ross  process. We underline their common and distinct properties. More precisely, we begin by presenting several results that provide upper and lower bounds for the time-asymptotic growth rates of both processes. These bounds exhibit notable similarities between the two models.

Next, we explore the approximation of CIR and squared Bessel processes by a sequence of CIR processes. We prove the convergence of this sequence in integral norms, assuming that the corresponding coefficients converge. Additionally, we establish upper bounds on the rate of convergence. It turns out that the CIR and squared Bessel processes are closely related, as the squared Bessel process can be represented as the limit of a sequence of CIR processes. However, as anticipated, the upper bounds involve coefficients that depend on the length of the time interval and tend to infinity as the interval length increases. In this sense, the processes diverge, or, in other words, they move apart. Nevertheless, the coefficients can be sufficiently close such that, over slowly increasing time intervals, the processes remain comparable.

We then apply this approximation to the problem of parameter estimation for the squared Bessel process using the maximum likelihood method. To establish the strong consistency of the constructed drift parameter estimator, we approximate the squared Bessel process by a sequence of CIR processes, for which the necessary convergence can be derived from their ergodic properties. Furthermore, we show how to estimate the diffusion coefficient of the process based on the realized quadratic variations.

Finally, we investigate both processes using the concept of stochastic instability. From this perspective, the properties of the squared Bessel and CIR processes are fundamentally different. We demonstrate that the squared Bessel process exhibits stochastic instability, whereas the CIR process is ergodic and, in this sense, stochastically stable. In addition, we consider an alternative sequence of approximations for the squared Bessel process and show that these approximating processes are also stochastically unstable. Moreover, we prove that, when appropriately normalized, this sequence converges to the (non-squared) Bessel process.

\subsection{Structure of the paper}
The remainder of the paper is organized as follows. In Section~\ref{sec:prelim}, we introduce the CIR and squared Bessel processes as unique solutions to the corresponding stochastic differential equations. We also provide preliminary information on their distributional and pathwise properties that are essential for the subsequent sections. In particular, this section presents formulas for their densities and the first three moments.

Section~\ref{sec:growth} contains several results that describe the growth rates of the CIR and squared Bessel processes as functions of time and their coefficients. Section~\ref{sec:distance} focuses on the distance between the CIR and squared Bessel processes in terms of integral norms, expressed through their coefficients. We establish the rate of convergence of CIR processes to either CIR or squared Bessel processes in integral norms over any fixed interval, under the condition that the respective coefficients converge.

Section~\ref{sec:statistics} addresses the statistical problem of identifying the Bessel process from continuous-time observations of its trajectory. Finally, in Section~\ref{sec:instability}, we demonstrate that the squared Bessel process is stochastically unstable, in contrast to the CIR process. This section also presents some functional limit theorems.

For the reader's convenience, several auxiliary results used in the proofs are provided in the appendix. Specifically, we include definitions and properties of special functions, a limit theorem for stochastic differential equations with non-Lipschitz diffusion terms, as well as additional limit theorems for solutions of stochastic differential equations.

\section{Preliminaries}
\label{sec:prelim}
Let $(\Omega, \FF, \pr)$ be a probability space and $W = \set{W_t, t \ge 0}$ be a Wiener process on it. In this paper we study two stochastic differential equations, namely
\begin{equation}\label{CIR}
X_t = x_0 + \int_0^t (a - b X_s)\,ds + \sigma \int_0^t \sqrt{X_s}\,dW_s,
\end{equation}
and
\begin{equation}\label{BES}
Y_t = y_0 + a t + \sigma \int_0^t \sqrt{Y_s}\,dW_s,
\end{equation}
where $x_0 > 0$, $y_0 > 0$, $a > 0$, $b \ge 0$, and $\sigma > 0$.

It is well known that both equations \eqref{CIR} and \eqref{BES} admit unique non-negative strong solutions,  $X = \{X_t, t \ge 0\}$ and $Y = \{Y_t, t \ge 0\}$, respectively.

The process $X = \{X_t, t \ge 0\}$ was introduced in \cite{CIR85} for the purpose of interest rates modeling. It is commonly referred to as the Cox--Ingersoll--Ross (CIR) process.
The process $Y = \{Y_t, t \ge 0\}$ is the squared Bessel process, see, e.g., \cite{GJY03} or \cite[Chapter XI]{RevuzYor} for details.

It follows from the comparison theorem \cite[Proposition 2.18, p.~293]{KaratzasShreve} that if $x_0 \le y_0$, then
\[
\pr\left(X_t \le Y_t \text{ for all } t\ge0\right) = 1.
\]

In what follows we additionally assume that
\[
2a \ge \sigma^2.
\]
In this case, the trajectories of  both  processes $X$ and $Y$ with probability 
 1 remain  strictly positive, while in the case $0 < a < \sigma^2/2$, they  almost surely hit  zero, where the state 0 is instantaneously reflecting (see, e.g., classical paper \cite{GJY03} and the more recent ones \cite{MiYu2, MIYu1} for more details).

\subsection{Distributional and path-wise properties of CIR process}
It is well known \cite{CIR85} that $X_t$ follows a non-central chi-squared distribution with the following probability density function:
\[
p_t(x) = \frac{1}{c(t)} \left(\frac{x}{x_0 e^{-bt}}\right)^{\nu/2} \exp\set{- \frac{x + x_0e^{-bt}}{c(t)}}
I_\nu\left(\frac{2 e^{-bt/2} \sqrt{x x_0}}{c(t)} \right) \ind_{x>0},
\]
where
\[
c(t) = \frac{\sigma^2}{2b}\left(1 - e^{-bt}\right),
\quad
\nu = \frac{2a}{\sigma^2} - 1,
\]
and $I_\nu$ is the modified Bessel function of the first kind.
For $\nu>-1$ and $x\in\real$ this function is defined by the following power series \cite[Formula~50:6:1]{An-Atlas-of-Functions}:
\[
I_\nu(x) = \sum_{j=0}^\infty \frac{(x/2)^{2j+\nu}}{j!\Gamma(j+1+\nu)},
\]
where $\Gamma$ stands for the Gamma function.
Obviously, $I_\nu(0) = 0$ for all $\nu>1$, more precisely $I_\nu$ has the following behavior as $x \to 0$:
\begin{equation*}
I_\nu(x) \sim \frac{(x/2)^{\nu}}{\Gamma(\nu+1)}.
\end{equation*}
Using this relation, one can show that
\begin{equation}\label{eq:gammadens}
p_t(x) \to \frac{(2b/\sigma^2)^{2a/\sigma^2}}{\Gamma(2a/\sigma^2)} \, x^{2a/\sigma^2 - 1}e^{- 2bx/\sigma^2}\ind_{x>0} 
\eqqcolon p_\infty(x),
\quad t \to \infty.
\end{equation}
Note that the limiting distribution is a Gamma distribution.

Moreover, the CIR process $X$ is ergodic \cite{CIR85} (see also \cite[Section 1.2]{Alfonsi15} and \cite{DMR22}).
Ergodicity implies that for any function $f\in L^1(\real,p_\infty(x)dx)$,
the time average
$\frac{1}{T}\int_0^Tf(X_t)dt$
converges a.s.\ to the space average
$\int_{\real} f(x)p_\infty(x)dx$,
as $T\to\infty$.
In particular, for $a > \frac{\sigma^2}{2}$,
\begin{equation}\label{ergo-2}
\frac{1}{T}\int_0^T \frac{dt}{X_t} \rightarrow \int_{\mathbb{R}} \frac{p_\infty(x)}{x}\,dx
= \frac{b}{a-\sigma^2/2},\quad\text{a.\,s., when } T\rightarrow\infty.
\end{equation}

The first two moments of $X_t$ are equal to
\begin{equation}\label{eq:CIR-mean}
\ex X_t = x_0 e^{-bt} + \frac{a}{b}\left(1 - e^{-b t}\right),
\end{equation}
and 
\begin{align}\ex X_t^2 = \frac{x_0 (\sigma^2 + 2 a)}{b} \left(e^{-bt} - e^{-2bt}\right) 
+ \frac{a(\sigma^2 + 2a)}{2b^2}\left(1 - e^{-b t}\right)^2
+ x_0^2 e^{-2bt}.
\label{eq:CIR-var}
\end{align}

The formula for higher moments of the CIR processes is presented in \cite[Proposition 1]{Okhrin23}. In particular,
\begin{align*}
\ex X_t^3 &= x_0^3 e^{-3bt} + \left(1 + \frac{3\sigma^2}{2a} + \frac{\sigma^4}{2a^2} \right) \left(\frac{a^3}{b^3} \left(1 - e^{-bt}\right)^3
+ \frac{3 x_0 a^2}{b^2}  \left(e^{-bt} - 2 e^{-2bt} + e^{-3bt}\right)\right)
\\
&\quad+  \frac{3 x_0^2 a}{b} \left(1 + \frac{\sigma^2}{a}\right)
\left(e^{-2bt} - e^{-3bt}\right).
\end{align*}

\subsection{Distributional and path-wise properties of squared Bessel process}
The probability density function of the squared Bessel process $Y_t$ is given by 
\[
g_t(x) = \frac12 \left(\frac{x}{y_0}\right)^{\nu/2}
\exp\set{-\frac{2(x+y_0)}{\sigma^2 t}} I_{\nu} \left(\frac{4 \sqrt{x y_0}}{\sigma^2 t}\right) \ind_{x>0},
\]
where, as before, $\nu = \frac{2a}{\sigma^2} - 1$ (see, e.g., \cite[Chapter XI, Corollary (1.4)]{RevuzYor}).

Unlike the CIR process $X$, the squared Bessel process $Y$ is non-ergodic. For a detailed discussion on the properties of squared Bessel processes, we refer the reader to \cite[Chapter XI]{RevuzYor}. A comparison of the properties of both ergodic and non-ergodic processes, $X$ and $Y$, can be found in \cite{BAK2013}.

Since $I_\nu(0) = 0$, we see that
\begin{equation}\label{denstozero}
g_t(x) \to 0, \quad t \to \infty.
\end{equation}
Therefore, for the squared Bessel process, the limiting distribution does not exist.

The first and second moments of $Y_t$ are equal to:
\begin{equation}\label{eq:Bessel-moments}
\ex Y_t = y_0 + a t,
\quad
\ex Y_t^2 = y_0^2 + \left(\frac{\sigma^2}{2} + a\right) \left(2y_0 t + a t^2\right).
\end{equation}
Both equalities can be derived directly from the equation \eqref{BES}, taking into account that the stochastic integral $\int_0^t \sqrt{Y_s}\,dW_s$ is a square-integrable martingale with zero mean whose second moment equals $\int_0^t \ex Y_s\,ds$.

\begin{remark}
We see from \eqref{eq:CIR-mean} and \eqref{eq:CIR-var} that the first two moments of the CIR process exist for all $t$.
Moreover, they are totally bounded. Indeed,  as it was established  in \cite[Proposition 3]{BAK2013},
$\sup_{t\ge 0} \ex X_t^p < \infty$ for any $p > -2a/\sigma^2$.
In contrast, the first and second moments of the squared Bessel process exhibit linear and quadratic growth with respect to $t$, respectively.
\end{remark}

The general formula for the moments of the Bessel process has the following form: for any $p \ge -\frac{2a}{\sigma^2}$
\[
\ex Y_t^p = \left(\frac{\sigma^2 t}{2}\right)^p \frac{\Gamma\left(\frac{2a}{\sigma^2} + p\right)}{\Gamma\left(\frac{2a}{\sigma^2}\right)}
\exp\set{-\frac{2y_0}{\sigma^2 t}} {}_1F_1 \left(\frac{2a}{\sigma^2} + p, \frac{2a}{\sigma^2},\frac{2y_0}{\sigma^2 t}\right),
\]
see the proof of Proposition 3 in \cite{BAK2013}. Here ${}_1F_1$ is the confluent hypergeometric function, see Appendix~\ref{app:special}.
Using formula \eqref{1F_1(c+3)} in Appendix, we can derive for $p=3$
\begin{equation}\label{BESmom3}
\ex Y_t^3 = \left(\frac{a\sigma^4}{2} + \frac{3a^2\sigma^2}{2} + a^3\right) t^3
+ 3\left(\frac{y_0 \sigma^4}{2} + \frac{3 a y_0\sigma^2}{2} + a^2y_0\right) t^2
+ 3 y_0^2\left(\sigma^2 + a\right) t + y_0^3.
\end{equation}

\section{Time-asymptotic growth rate for CIR and quadratic Bessel processes}
\label{sec:growth}

Now we establish   several  results that provide  a growth rate for the solution to  equations \eqref{BES} and \eqref{CIR}, as the function of   time and coefficients. As usual,  time is included in the constants, since the time interval is fixed in many situations, but for us it is the asymptotic behaviour of functionals of solutions that is most important. We demonstrate what growth rates can be obtained by different methods, and compare the results. The first  result follows from the Gr\"{o}nwall inequality therefore gives an exponential  growth rate. This growth rate is determined by   coefficients $a$, $x_0$ and $ \sigma$ and is valid both for solution to \eqref{CIR} or \eqref{BES}.

\begin{proposition}
      \label{l:CIR-growth}
Let $Z = \{Z_t, t \ge 0\}$ be a unique solution to \eqref{CIR} or \eqref{BES} (i.e., $Z = X$ or $Z = Y$). Then, for all $t \ge 0$,
\begin{equation}\label{CIR-growth1}
\ex\left(\sup_{s \le t} Z_s\right)^2
\le 2\left(\left(x_0 + a t\right)^2 + 2\sigma^2 t\right) e^{4\sigma^2 t}.
\end{equation}
\end{proposition} 

\begin{proof}
Consider the equation \eqref{CIR}, i.e., $Z=X$.
Define 
\[
\tau_N \coloneqq \inf\{t \ge 0 : X_t \ge N\}, \quad N \ge 1.
\]
Then
\begin{align}
0\le X_{t \wedge \tau_N} &= x_0 + a (t \wedge \tau_N) - b \int_0^{t \wedge \tau_N} X_u\,du + \sigma \int_0^{t \wedge \tau_N} \sqrt{X_u}\,dW_u
\notag\\
&\le x_0 + a (t \wedge \tau_N) + \sigma \int_0^{t \wedge \tau_N} \sqrt{X_u}\,dW_u,
\label{CIR4terms}
\end{align}
which implies
\begin{equation}\label{CIR2-4terms}
\ex\left(\sup_{s \le t} X_{s \wedge \tau_N}\right)^2
\le 2\left(\left(x_0 + a t\right)^2 + \sigma^2 \ex \sup_{s \le t}  \left(\int_0^{s \wedge \tau_N} \sqrt{X_u}\,dW_u\right)^2 \right).
\end{equation}
Since
\[
\ex \left(\int_0^{t \wedge \tau_N} \sqrt{X_s}\,dW_s\right)^2
= \ex \left(\int_0^t \ind_{s \le \tau_N} \sqrt{X_s}\,dW_s\right)^2
= \int_0^t \ex \left(\ind_{s \le \tau_N} \sqrt{X_s}\right)^2 ds
\le N t,
\]
we see that $\int_0^{t \wedge \tau_N} \sqrt{X_s}\,dW_s$ is a square-integrable martingale. Therefore, by the Doob maximal quadratic inequality,
\begin{align}
\ex \sup_{s \le t} \left(\int_0^{s \wedge \tau_N} \sqrt{X_u}\,dW_u\right)^2
&\le 4 \ex \int_0^{t} X_{s\wedge \tau_N}\,ds 
\le 2 \ex \int_0^{t} \left(1 + X_{s\wedge \tau_N}^2\right)\,ds \notag \\
&\le 2t + 2 \int_0^{t} \ex\left(\sup_{u \le s} X_{u \wedge \tau_N}\right)^2\,ds.
\label{CIR2-term4}
\end{align}
Combining \eqref{CIR2-4terms}--\eqref{CIR2-term4}, we derive the inequality
\[
\ex\left(\sup_{s \le t} X_{s \wedge \tau_N}\right)^2
\le 2\left(\left(x_0 + a t\right)^2 +  2\sigma^2 t + 2\sigma^2 \int_0^{t} \ex\left(\sup_{u \le s}X_{u \wedge \tau_N}\right)^2\,ds \right).
\]
By the Gr\"{o}nwall  inequality,
\[
\ex\left(\sup_{s \le t} X_{s \wedge \tau_N}\right)^2
\le 2\left(\left(x_0 + a t\right)^2 +  2\sigma^2 t\right) e^{4\sigma^2 t}.
\]
Letting $N \to \infty$ concludes the proof for $X$.
The equation \eqref{BES} is considered similarly.
\end{proof}

\begin{remark}
Similarly to the above proof, one can establish that for the solution $X = \{X_t, t \ge 0\}$ of the equation \eqref{CIR}, the following inequality holds:
\[
\ex \sup_{s\le t} \left(X_s + b \int_0^s X_u\, du \right)^2
\le 2\left(\left(x_0 + a t\right)^2 +  2\sigma^2 t\right) e^{4\sigma^2 t}.
\]
The disadvantage of these upper bounds is that being exponential, they grow quickly in time $t$, but their advantage is that they do not depend on coefficient $b$. 
\end{remark}
Now our goal is to obtain the upper bounds for $\ex \sup_{s\le t} \left(X_s + b \int_0^s X_u\, du \right)^2$ that will grow not so quickly.  We will not apply the Gr\"{o}nwall  inequality that always gives exponential growth, but apply  directly the values of the moments of CIR-process. Oppositely to previous bounds, the next ones will depend on $b$ and the next goal will be to analyze this dependence when $b\downarrow 0.$ 
\begin{proposition}\label{l:CIR-growth2}
Let $X = \{X_t, t \ge 0\}$ be a solution to \eqref{CIR}. Then, for all $t \ge 0$,
\begin{equation}\label{CIR-bound2}
\ex \sup_{s\le t} \left(X_s + b \int_0^s X_u\, du \right)^2
\le 2(x_0 + a t)^2 +  \frac{8\sigma^2}{b} \left(x_0 -\frac{a}{b}\right)\left(1 - e^{-b t}\right) + \frac{8\sigma^2 a}{b} t.
\end{equation}
\end{proposition}

\begin{proof}
Similarly to the proof of Proposition~\ref{l:CIR-growth}, we get
\begin{align*}
\ex \sup_{s\le t} \left(X_s + b \int_0^s X_u\, du \right)^2
&\le 2(x_0 + a t)^2 + 2\sigma^2 \ex \sup_{s \le t} \left(\int_0^s \sqrt{X_u}\,dW_u\right)^2
\\
&\le 2(x_0 + a t)^2 + 8\sigma^2 \int_0^t \ex X_s\,ds.
\end{align*}
Integrating \eqref{eq:CIR-mean}, we obtain
\begin{equation}\label{eq:CIR-intmean}
\int_0^t \ex X_s\,ds  = \frac{1}{b} \left(x_0 -\frac{a}{b}\right)\left(1 - e^{-b t}\right) + \frac{a}{b} t.
\end{equation}
Combining these two formulas we conclude the proof.
\end{proof}

\begin{remark}
It follows from \eqref{eq:CIR-var} and \eqref{eq:CIR-intmean} that
 \begin{align*}
\MoveEqLeft
\ex \sup_{s\le t} \left(X_s + b \int_0^s X_u\, du \right)^2
\ge \ex X_t^2 + b^2 \ex \left(\int_0^t X_u\, du \right)^2
\ge \ex X_t^2 + b^2 \left(\int_0^t \ex X_u\, du \right)^2
\\
&= \frac{x_0 (\sigma^2 + 2 a)}{b} \left(e^{-bt} - e^{-2bt}\right) 
+ \frac{a(\sigma^2 + 2a)}{2b^2}\left(1 - e^{-b t}\right)^2
+ x_0^2 e^{-2bt}
\\
&\quad + \left(\left(x_0 -\frac{a}{b}\right)\left(1 - e^{-b t}\right) 
+ a t\right)^2
 \end{align*}
Comparing this lower bound with the upper bound in \eqref{CIR-bound2}, we observe that both the upper and lower bounds for $\ex \sup_{s \le t} (X_s + b \int_0^s X_u, du)^2$ exhibit a quadratic rate of growth, as $t \to \infty$. The upper bound follows the asymptotic behavior $2a^2 t^2 + 4a (x_0 + \frac{2\sigma^2}{b}) t + O(1)$, while the lower bound behaves as $a^2 t^2 + 2a (x_0 - \frac{a}{b}) t + O(1)$.
\end{remark}

\begin{remark}
\label{rem:BES-bounds}
\begin{itemize}
 \item[$(i)$] Arguing as in the proof of Proposition~\ref{l:CIR-growth2} and using \eqref{eq:Bessel-moments} instead of \eqref{eq:CIR-mean}, we get the bound
 \begin{equation}\label{besselas}
     \ex \sup_{s\le t} Y_s^2 \le
    2(y_0 + a t)^2 +   4\sigma^2 \left(2 y_0 t + a t^2\right).
    \end{equation}
    Note that the right-hand side of \eqref{besselas} is a limit, as $b \downarrow 0$, of the right-hand side of \eqref{CIR-bound2}. This can be easily seen from the relation
    $1 - e^{-bt} = bt - \frac{b^2t^2}{2} + o (b^3)$, $b \downarrow 0$. However, if to fix $b$ and consider the asymptotic behaviour, as $t\to\infty$, of the right-hand sides of \eqref{CIR-bound2} and \eqref{besselas}, we see that the main part of the right-hand side  of \eqref{CIR-bound2} equals $2a^2t^2$ while the main part of the right-hand side  of \eqref{besselas} equals $(2a^2+4a\sigma^2)t^2$. It means (a bit unexpectedly) that the difference between  asymptotic behaviour  of $\ex \sup_{s\le t} Y_s^2$ and $\ex \sup_{s\le t} X_s^2$ is very significantly determined by the diffusion coefficient  $\sigma$ as well as (more expected) of the drift coefficient $a$.   Of course, this difference in some latent way depends on $b$ because the value $\frac{8\sigma^2}{b} \left(x_0 -\frac{a}{b}\right)\left(1 - e^{-b t}\right)$ in the right-hand side of \eqref{CIR-bound2} is bounded in $t$ for any $b>0$, however, it is growing as $t^2$ if we come to the limit, as $b\to 0$. 

\item[$(ii)$] The second formula in \eqref{eq:Bessel-moments} implies the following lower bound:
\begin{equation}\label{bessel-low}
     \ex \sup_{s\le t} Y_s^2 \ge \ex Y_t^2
    = (y_0 + a t)^2 +  \frac{\sigma^2}{2} \left(2 y_0 t + a t^2\right).
\end{equation}
 It can be observed that, compared to the upper bound in \eqref{besselas}, this lower bound contains the same terms, but with smaller coefficients. Therefore, we conclude that $\ex \sup_{s \le t} Y_s^2$ grows quadratically as a function of $t$.

\item[$(iii)$] 
Let us compare two upper bounds for the squared Bessel process, specifically \eqref{CIR-growth1} and \eqref{besselas}. On the one hand, the bound given by \eqref{besselas} is clearly more advantageous for large values of $t$. On the other hand, when $t$ is near zero, the bound \eqref{CIR-growth1} may provide greater accuracy. More precisely, the bound \eqref{CIR-growth1} is superior to \eqref{besselas} if and only if $e^{4\sigma^2 t} \leq 2x_0 + a t$. If $x_0 > 1/2$, this condition is satisfied for sufficiently small values of $t$.
\end{itemize}
\end{remark}

\section{Distance between CIR and squared Bessel processes in integral norms}
\label{sec:distance}
In this section we establish the rate of convergence of CIR processes and also CIR processes to squared Bessel process in two integral norms, $L_1([0,T], \mathbf{P})$ and $L_2([0,T], \mathbf{P})$   on any fixed interval $[0,T]$ under the condition that the respective coefficients converge. In this sense, CIR and squared Bessel processes are close. However, as expected, the upper bounds of the distance between them contain coefficients depending on the length of the interval and tending to $\infty$, as the length tends to $\infty$. In this sense, the processes disperse, or, in other words, move away. Despite this fact,   the coefficients can be  so close that, under slow growth of time interval,  the processes can be still close.

So, consider the following sequence of stochastic differential equations:
\begin{equation}\label{CIRn}
X_n(t) = x_0 + \int_0^t (a_n - b_n X_n(s))\,ds + \sigma_n \int_0^t \sqrt{X_n(s)}\,dW_s,
\quad n \geq 0,
\end{equation}
where $x_0 > 0$, $a_n > 0$, $b_n \ge 0$,  and $\sigma_n > 0$  for all $n \geq 0$.

Assume that 
\[
a_n \to a_0, \quad b_n \to b_0, \quad \sigma_n \to \sigma_0,
\quad \text{as } n \to \infty.
\]

Note that the equations in \eqref{CIRn} satisfy conditions (Y1$_n$)--(Y4$_n$) from Section 4 of \cite{MPS09}, see Appendix~\ref{app:lim-thm}. According to Theorem \ref{th:MPS09}, we have that for any $T > 0$
\[
\sup_{0 \le t \le T} \ex \abs{X_n(t) - X_0(t)} \to 0,
\quad \text{as } n \to \infty.
\]
 In the following theorem, we establish an upper bound for  rate of convergence.

\begin{theorem}\label{rateconv}
\begin{itemize}
\item[$(i)$] Let $b_0 > 0$. Then for any $T > 0$, the following upper bound holds:
\begin{equation}\label{CIRrate}
\sup_{t\in[0,T]} \ex \abs{X_n(t) - X_0(t)} 
\leq e^{b_n T} \Bigl( \abs{a_n - a_0} T + \abs{b_n - b_0} A_0^2(T) + \abs{\sigma_n - \sigma_0} A_0(T) \Bigr),
\end{equation}
where
\begin{equation}\label{A(T)}
A_0^2(T) \coloneqq \frac{1}{b_0} \left(x_0 -\frac{a_0}{b_0}\right)\left(1 - e^{-b_0 T}\right) + \frac{a_0}{b_0} T.
\end{equation}

\item[$(ii)$]
Let $b_0 = 0$. Then for any $T > 0$, the following upper bound holds:
\begin{equation}
\sup_{t\in[0,T]} \ex \abs{X_n(t) - X_0(t)} 
\leq e^{b_n T} \Bigl( \abs{a_n - a_0} T + b_n B_0^2(T) + \abs{\sigma_n - \sigma_0} B_0(T) \Bigr),
\label{BESrate}
\end{equation}
where 
\[
B_0^2(T) = x_0 T + \frac{a_0}{2}\, T^2.
\]
\end{itemize}
\end{theorem}

\begin{remark}
The function $A_0^2(T)$, defined by \eqref{A(T)}, is positive, since
$A_0^2(T) = \int_0^T \ex X_0 (s)\, ds$, according to \eqref{eq:CIR-intmean}.
Moreover, it exhibits linear growth and satisfies the following bounds:
\[
\min\set{x_0, \frac{a_0}{b_0}} T \le A_0^2(T) \le \max\set{x_0, \frac{a_0}{b_0}} T. 
\]
To verify these inequalities, we consider two cases.

Case 1: $x_0 - \frac{a_0}{b_0} \ge 0$. Using the inequality 
$0 \le 1 - e^{-b_0 T} \le b_0 T$, we obtain that
\begin{equation}\label{A-bounds}
\frac{a_0}{b_0} T \le A_0^2(T) \le \frac{1}{b_0} \left(x_0 -\frac{a_0}{b_0}\right) b_0 T + \frac{a_0}{b_0} T
= x_0 T.
\end{equation}

 Case 2: $x_0 - \frac{a_0}{b_0} < 0$.
In this case, both inequalities in \eqref{A-bounds} are reversed.
\end{remark}




\begin{proof}[Proof of Theorem~\ref{rateconv}]
$(i)$
Define an auxiliary process 
\[
\widetilde{X}_n(t) = x_0 + \int_0^t (a_n - b_n X_0(s))\,ds + \sigma_n \int_0^t \sqrt{X_0(s)}\,dW_s.
\]
Then using the Cauchy--Schwarz inequality and the It\^o isometry, we get
\begin{align} 
\MoveEqLeft
\ex \abs{\widetilde{X}_n(t) - X_0(t)} \notag \\
&= \abs{a_n - a_0} t + \abs{b_n - b_0} \int_0^t \ex X_0(s)\,ds + \abs{\sigma_n - \sigma_0} \ex\abs{\int_0^t \sqrt{X_0(s)}\,dW_s} \notag \\
&\le \abs{a_n - a_0} t + \abs{b_n - b_0} \int_0^t \ex X_0(s)\,ds + \abs{\sigma_n - \sigma_0} \left(\int_0^t \ex X_0(s)\,ds\right)^{\frac12}.
\label{CIRtilde}
\end{align}
Note that, by \eqref{eq:CIR-intmean},
\begin{equation}\label{intmeanCIR}
\int_0^t \ex X_0(s)\,ds \le \int_0^T \ex X_0(s)\,ds = A_0^2(T),
\end{equation}
therefore, \eqref{CIRtilde} becomes
\begin{equation}\label{CIRtilde1}
\ex \abs{\widetilde{X}_n(t) - X_0(t)} 
\le \abs{a_n - a_0} T + \abs{b_n - b_0} A_0^2(T) + \abs{\sigma_n - \sigma_0} A_0(T).
\end{equation}

Now let us follow the Yamada method from \cite{YamadaWatanabe}. Define
\[
\alpha_m = \exp\set{-\frac{m(m+1)}{2}}, \quad m \ge 0.
\]
It follows that
$1 = \alpha_0 > \alpha_1 > \alpha_2 > \dots > 0$
and
\[
\int_{\alpha_m}^{\alpha_{m-1}}\frac{1}{x}\,dx = m, \quad m \ge 1.
\]
Thus, for every $m\ge 1$, there exists a continuous function  $\psi_m\colon \real \to \real$ with compact support in $(\alpha_m, \alpha_{m-1})$ such that
\[
0 \le \psi_m (x) \le \frac{2}{m x}, \; x\in (\alpha_m, \alpha_{m-1}), 
\quad\text{and}\quad
\int_0^\infty \psi_m (x) dx = 1.
\]
Next, we define
\[
\varphi_m(x) = \int_0^{\abs{x}}\!\!\int_0^y \psi_m(y) \,dy\,dx,
\quad x \in \real.
\]
The function $\varphi_m$ satisfies the following properties:
\[
0 \le \varphi_m(x) \le \abs{x}
\quad\text{and}\quad
\abs{\varphi_m'(x)} \le 1,
\]
and it is twice continuously differentiable since
$\varphi_m''(x) = \psi_m(\abs{x})$.

Moreover, by the Lebesgue dominated convergence theorem,
$\lim_{m\to\infty}\varphi_m(x) = \abs{x}$,
since
$\lim_{m\to\infty} \int_0^y \psi_m(y) \,dy = 1$.

By applying the It\^o formula, we obtain
\begin{align*}
\MoveEqLeft
\ex \varphi_m \left(X_n(t) - \widetilde X_n(t)\right)
\\
&= - b_n\int_0^t \ex\left[\left(X_n(s) - X_0(s)\right) 
\varphi_m'\left(X_n(s) - X_0(s)\right)\right] ds
\\
&\quad+ \frac{\sigma_n^2}{2} \int_0^t \ex\left[\left(\sqrt{X_n(s)} - \sqrt{X_0(s)}\right)^2 \varphi_m''\left(X_n(s) - X_0(s)\right)\right]ds
\\
&\le b_n \int_0^t \ex\abs{X_n(s) - X_0(s)} ds
+ \frac{\sigma_n^2}{2} \int_0^t \ex\left[\abs{X_n(s) - X_0(s)} \, \frac{2}{m \abs{X_n(s) - X_0(s)}}\right]ds
\\
&= b_n \int_0^t \ex\abs{X_n(s) - X_0(s)} ds
+ \frac{\sigma_n^2t}{m},
\end{align*}
Here, we have used the facts that $|\varphi_m'(x)|\le 1$, $\varphi_m''(x) = \psi_m(|x|) \le \frac{2}{m|x|}$, and 
$|\sqrt{x} - \sqrt{y}| \le \sqrt{\abs{x-y}}$.
By taking the limit, as $m\to\infty$, and applying Fatou's lemma, we obtain
\[
\ex \abs{X_n(t) - \widetilde X_n(t)}
\le b_n \int_0^t \ex\abs{X_n(s) - X_0(s)} ds.
\]
Then
\begin{align*}
\sup_{t \in [0,T]} \ex \abs{X_n(t) - X_0(t)}
&\le \sup_{t \in [0,T]} \ex \abs{\widetilde X_n(t) - X_0(t)}
+\sup_{t \in [0,T]} \ex \abs{X_n(t) - \widetilde X_n(t)}
\\
&\le \sup_{t \in [0,T]} \ex \abs{\widetilde X_n(t) - X_0(t)}
+ b_n \int_0^T \sup_{u \in [0,s]} \ex \abs{X_n(u) - X_0(u)} ds.
\end{align*}
Applying the Gr\"onwall inequality yields
\[
\sup_{t \in [0,T]} \ex \abs{X_n(t) - X_0(t)}
\leq e^{b_n T} \sup_{t \in [0,T]} \ex \abs{\widetilde{X}_n(t) - X_0(t)}.
\]
Finally,  utilizing the bound from \eqref{CIRtilde1}, we arrive at \eqref{CIRrate}.

$(ii)$
The proof of \eqref{BESrate} follows similar steps as the proof of \eqref{CIRrate}. In this case, instead of \eqref{intmeanCIR}, we have
\[
\int_0^T \ex X_0(s) \,ds = x_0 T + \frac{a_0}{2}\, T^2 = B_0^2(T),
\]
which is derived from \eqref{eq:Bessel-moments}.
\end{proof}

\begin{remark}
Note that, as $T\to \infty$,
\[
A_0^2(T) \sim \frac{a_0}{b_0} T,
\quad
B_0^2(T) \sim \frac{a_0}{2} T^2.
\]
Hence, as $T\to\infty$, the right-hand sides of \eqref{CIRrate} and \eqref{BESrate} are asymptotically equivalent to
\[
\frac{1}{b_0}\bigl(b_0\abs{a_n - a_0} + a_0 \abs{b_n - b_0}\bigr) T e^{b_n T}
\quad\text{and}\quad
\frac{1}{2} a_0 b_n T^2 e^{b_n T}
\]
respectively.
\end{remark}

If, in the above theorem, we take $a_n \equiv a$, $\sigma_n \equiv \sigma$, and $b_0 = 0$, the resulting bound simplifies significantly. Specifically, we  obtain the following result concerning the approximation of the squared Bessel process by a sequence of CIR processes.
\begin{corollary}\label{approxCIR}
Let $Y$ be a solution to \eqref{BES}. Consider a sequence of the stochastic differential equations:
\[
Y_n(t) = y_0 + \int_0^t (a - b_n Y_n(s))\,ds + \sigma \int_0^t \sqrt{Y_n(s)}\,dW_s,
\quad n \geq 1,
\]
where $b_n\downarrow 0$, $n\to\infty$.
Then, for any $T > 0$, the following bound holds:
\[
\sup_{t\in[0,T]} \ex \abs{Y_n(t) - Y(t)}\leq e^{b_n T} b_n T \left(x_0 + \tfrac12 a T\right).
\]
\end{corollary}

\begin{corollary}
Let $b_n\downarrow0$, $n\to\infty$, and let the sequence $\{T_n, n\ge1\}$ satisfy
\begin{equation}\label{bn-Tn}
T_n \to \infty, \quad e^{b_n T_n} b_n T_n^2 \to 0, \quad 
\text{as } n\to\infty.
\end{equation}
Then
\[
\sup_{t\in[0,T_n]} \ex \abs{Y_n(t) - Y(t)}
 \to 0, \quad n\to\infty.
\]
\end{corollary}
\begin{remark}
For example, the condition \eqref{bn-Tn} is fulfilled for $b_n = 1/n$, $T_n = \log (\log n)$, $n \ge 2$.
\end{remark}
\begin{remark} If to analyze the proof of Theorem \ref{rateconv}, it consists of two parts: obtaining  inequality \eqref{CIRtilde} and applying of Yamada method. If inequality of the form \eqref{CIRtilde} can be obtained for the second moment $\ex \left( \widetilde{X}_n(t) - X_0(t) \right)^2$, Yamada method is more involved and allows to obtain just the first moment $\ex \left| \widetilde{X}_n(t) - X_0(t) \right|.$ However, we can obtain the upper bound for $\ex \left(  {X}_n(t) - X_0(t) \right)^2$ as  well, as the next result states.
\end{remark}

\begin{theorem}\label{rateconv2}
\begin{itemize}
\item[$(i)$] Let $b_n > 0$ for all $n\ge0$. Then for any $T > 0$, the following upper bound holds:
\begin{multline}\label{CIRbound2}
\sup_{t\in[0,T]}\ex \left(X_n(t) - X_0(t) \right)^2 \\
\le 2 \left(R_n(T) + R_0(T)\right)^{\frac12} e^{\frac{b_n T}{2}} \Bigl( \abs{a_n - a_0} T + \abs{b_n - b_0} A_0^2(T) + \abs{\sigma_n - \sigma_0} A_0(T) \Bigr)^{\frac12},
\end{multline}
where $A_0(T)$ is defined in Theorem~\ref{rateconv} and
\begin{align*}
R_n(T) &= x_0^3 + \left(1 + \frac{3\sigma_n^2}{2a_n} + \frac{\sigma_n^4}{2a_n^2} \right) \left(\frac{a_n^3}{b_n^3} \left(1 - e^{-b_n T}\right)^3
+ \frac{3 x_0 a_n^2}{b_n^2} \left( 1 - e^{-b_nT}\right)^2\right)
\\
&\quad+ \frac{3 x_0^2 a_n}{b_n} \left(1 + \frac{\sigma_n^2}{a_n}\right)
\left(1 - e^{-b_nT}\right).
\end{align*}
\item[$(ii)$] Let $b_0 = 0$ and $b_n > 0$ for $n\ge1$. Then for any $T > 0$, the following upper bound holds:
\begin{multline*}
\sup_{t\in[0,T]}\ex \left(X_n(t) - X_0(t) \right)^2 \\
\le 2 \left(R_n(T) + \widetilde R_0(T)\right)^{\frac12} e^{\frac{b_n T}{2}}\Bigl( \abs{a_n - a_0} T + b_n B_0^2(T) + \abs{\sigma_n - \sigma_0} B_0(T) \Bigr)^{\frac12},
\end{multline*}
where $B_0(T)$ is defined in Theorem~\ref{rateconv} and
\begin{align*}
\widetilde R_0(T) &=\left(\frac{a_0\sigma_0^4}{2} + \frac{3a_0^2\sigma_0^2}{2} + a_0^3\right) T^3
+ 3\left(\frac{x_0 \sigma_0^4}{2} + \frac{3 a_0 x_0\sigma_0^2}{2} + a_0^2x_0\right) T^2
+ 3 x_0^2\left(\sigma_0^2 + a_0\right) T + x_0^3.
\end{align*}
\end{itemize}
\end{theorem}
\begin{proof}
$(i)$ By the Cauchy--Schwarz inequality,
\begin{align*}
\ex \left(X_n(t) - X_0(t) \right)^2
&\le \ex \abs{X_n(t) - X_0(t)}^{\frac12} \abs{X_n(t) - X_0(t)}^{\frac32}\\
&\le \left(\ex \abs{X_n(t) - X_0(t)}\right)^{\frac12} \left(\ex \abs{X_n(t) - X_0(t)}^3\right)^{\frac12}.
\\
&\le 2 \left(\ex \abs{X_n(t) - X_0(t)}\right)^{\frac12} \left(\ex X_n^3(t) + \ex X_0^3(t)\right)^{\frac12},
\end{align*}
where the elementary inequality 
$(x+y)^3 \le 4(x^3 + y^3)$
has been used.

Further, for any $n\ge0$ and any $t \in [0,T]$,
\begin{align*}
\ex X_n^3(t) &= x_0^3 e^{-3b_nt} + \left(1 + \frac{3\sigma_n^2}{2a_n} + \frac{\sigma_n^4}{2a_n^2} \right) \left(\frac{a_n^3}{b_n^3} \left(1 - e^{-b_nt}\right)^3
+ \frac{3 x_0 a_n^2}{b_n^2} e^{-b_nt} \left( 1 - e^{-b_nt}\right)^2\right)
\\
&\quad+  \frac{3 x_0^2 a_n}{b_n} \left(1 + \frac{\sigma_n^2}{a_n}\right)
e^{-2b_nt} \left(1 - e^{-b_nt}\right) \le R_n(T).
\end{align*}
Then using the bound \eqref{CIRrate}, we obtain
\begin{multline*}
\sup_{t\in[0,T]}\ex \left(X_n(t) - X_0(t) \right)^2
\le 2 \left(R_n(T) + R_0(T)\right)^{\frac12} \left(\sup_{t\in[0,T]}\ex \abs{X_n(t) - X_0(t)}\right)^{1/2}
\\
\le 2 \left(R_n(T) + R_0(T)\right)^{\frac12} e^{b_n T/2} \Bigl( \abs{a_n - a_0} T + \abs{b_n - b_0} A_0^2(T) + \abs{\sigma_n - \sigma_0} A_0(T) \Bigr)^{\frac12}.
\end{multline*}

The statement $(ii)$ is derived from the second statement of Theorem \ref{rateconv}  in a similar manner, taking into account that 
$\sup_{t\in[0,T]} \ex X_0^3 \le \widetilde R(T)$ for $b_0 = 0$,
according to the formula \eqref{BESmom3} for the third moment of the squared Bessel process.
\end{proof}
Now we can establish the upper bound for the same value, $\sup_{t\in[0,T]} \ex (X_n(t) - X_0(t))^2$, as in the previous theorem, but using not knowledge of probability distributions, but methods of stochastic analysis.

\begin{theorem}\label{rateconv3}
\begin{itemize}
\item[$(i)$]
Let $b_n > 0$ for all $n\ge0$. Then for any $T > 0$, the following upper bound holds:
\begin{align}
&\sup_{t\in[0,T]} \ex (X_n(t) - X_0(t))^2 
\notag\\
&\le e^{(b_n+b_0) T} \left(2 \abs{a_n - a_0} + \sigma_0^2 + 2\sigma_0 \abs{\sigma_n -\sigma_0}\right) \left( \abs{a_n - a_0} T + \abs{b_n - b_0} A_0^2(T) + \abs{\sigma_n - \sigma_0} A_0(T) \right)  T 
\notag\\
&\quad + e^{b_0 T} \left(2 \abs{b_n - b_0} T\bigl(D_n^2(T) + D_n(T)D_0(T)\bigr) + \abs{\sigma_n -\sigma_0}^2 A_n^2(T)\right),
\label{CIRbound3}
\end{align}
where
\begin{align*}
A_n^2(T) &= \int_0^T \ex X_n(s)\,ds =\frac{1}{b_n} \left(x_n -\frac{a_n}{b_n}\right)\left(1 - e^{-b_n T}\right) + \frac{a_n}{b_n} T,
\\
D_n^2(T) &= \frac{x_0 (\sigma_n^2 + 2 a_n)}{b_n} \left(1 - e^{-bT}\right) 
+ \frac{a(\sigma_n^2 + 2a_n)}{2b_n^2}\left(1 - e^{-b T}\right)^2
+ x_0^2.
\end{align*}

\item[$(ii)$] Let $b_0 = 0$ and $b_n > 0$ for $n\ge1$. Then for any $T > 0$, the following upper bound holds:
\begin{align*}
\MoveEqLeft
\sup_{t\in[0,T]} \ex (X_n(t) - X_0(t))^2 
\\
&\le \left(2 \abs{a_n - a_0} + \sigma_0^2 + 2\sigma_0 \abs{\sigma_n -\sigma_0}\right) \left( \abs{a_n - a_0} T + b_n B_0^2(T) + \abs{\sigma_n - \sigma_0} B_0(T) \right)  T e^{b_n T}
\notag\\
& + 2 b_n T\bigl(D_n^2(T) + D_n(T)E_0(T)\bigr) + \abs{\sigma_n -\sigma_0}^2 A_n^2(T).
\end{align*}
where $B_0(T)$ is defined in Theorem \ref{rateconv} and
\begin{align*}
E_0^2(T) &= x_0^2 + \left(\frac{\sigma_0^2}{2} + a_0\right) \left(2x_0 T + a_0 T^2\right).
\end{align*}
\end{itemize}
\end{theorem}

\begin{proof}
$(i)$ Let us apply the It\^o formula to the function $F(x) = x^2$ and the process 
\begin{align*}
X_n(t) - X_0(t) &=  (a_n - a_0) t + (b_n - b_0) \int_0^t X_n(s)\,ds 
+ b_0 \int_0^t \bigl( X_n(s) - X_0(s)\bigr)\,ds
\\
&\quad + (\sigma_n -\sigma_0) \int_0^t \sqrt{X_n(s)}\,dW_s
+ \sigma_0 \int_0^t \left(\sqrt{X_n(s)} - \sqrt{X_0(s)}\right)\,dW_s.
\end{align*}
We get that
\begin{align}
(X_n(t) - X_0(t))^2 &= 2 \int_0^t (X_n(s) - X_0(s))\, d(X_n(s) - X_0(s))
\notag\\
&\quad + \sigma_0^2 \int_0^t \left(\sqrt{X_n(s)} - \sqrt{X_0(s)}\right)^2 ds
+ (\sigma_n -\sigma_0)^2 \int_0^t X_n(s)\,ds
\notag\\
&\quad+ 2\sigma_0 (\sigma_n -\sigma_0) \int_0^t \sqrt{X_n(s)} \left(\sqrt{X_n(s)} - \sqrt{X_0(s)}\right) ds.
\label{sqdiff}
\end{align}
Taking expectation, which is zero for stochastic integrals, we obtain from \eqref{sqdiff} that
\begin{align}
\MoveEqLeft
\ex (X_n(t) - X_0(t))^2 = 2 (a_n - a_0) \int_0^t \ex (X_n(s) - X_0(s))\, ds
\notag\\
& + 2 (b_n - b_0) \int_0^t \ex \bigl[(X_n(s) - X_0(s)) X_n(s)\bigr]\, ds
+ b_0 \int_0^t \ex (X_n(s) - X_0(s))^2\, ds
\notag\\
&+ \sigma_0^2 \int_0^t \ex\left(\sqrt{X_n(s)} - \sqrt{X_0(s)}\right)^2 ds
+ (\sigma_n -\sigma_0)^2 \int_0^t \ex X_n(s)\,ds
\notag\\
&+ 2\sigma_0 (\sigma_n -\sigma_0) \int_0^t \ex\left[\sqrt{X_n(s)} \left(\sqrt{X_n(s)} - \sqrt{X_0(s)}\right)\right] ds.
\label{exsqdiff}
\end{align}
The expectations in the right-hand side of \eqref{exsqdiff} can be bounded as follows:
\begin{gather*}
\ex \bigl\lvert (X_n(s) - X_0(s)) X_n(s)\bigr\rvert 
\le \ex X_n^2(s) + \ex X_0(s) X_n(s)
\le \ex X_n^2(s) + \sqrt{\ex X_0^2(s) \ex X_n^2(s)},
\\
\ex\left(\sqrt{X_n(s)} - \sqrt{X_0(s)}\right)^2 
\le \ex \abs{X_n(s) - X_0(s)},
\shortintertext{and}
\ex\left[\sqrt{X_n(s)} \left(\sqrt{X_n(s)} - \sqrt{X_0(s)}\right)\right]
\le \ex \abs{X_n(s) - X_0(s)},
\end{gather*}
where we have used inequalities
\[
\abs{\sqrt{x} - \sqrt{y}} \le \sqrt{\abs{x-y}}, \quad
\sqrt{x}\abs{\sqrt{x} - \sqrt{y}}
= \frac{\sqrt{x}}{\sqrt{x} + \sqrt{y}}\abs{x-y}
\le \abs{x-y}.
\]
Hence, \eqref{exsqdiff} implies
\begin{align}
\MoveEqLeft
\ex (X_n(t) - X_0(t))^2 \le \left(2 \abs{a_n - a_0} + \sigma_0^2 + 2\sigma_0 \abs{\sigma_n -\sigma_0}\right) \int_0^t \ex \abs{X_n(s) - X_0(s)}\, ds
\notag\\
& + 2 \abs{b_n - b_0}\int_0^t \left(\ex X_n^2(s) + \sqrt{\ex X_0^2(s) \ex X_n^2(s)}\right) ds
+ b_0 \int_0^t \ex (X_n(s) - X_0(s))^2\, ds
\notag\\
& + \abs{\sigma_n -\sigma_0}^2 \int_0^t \ex X_n(s)\,ds.
\label{exsqdiff1}
\end{align}
Now, the expectation $\ex \abs{X_n(s) - X_0(s)}$ is the first term in the right-hand side of \eqref{exsqdiff1} can be bounded with the help of Theorem~\ref{rateconv}, the expectations $\ex X_n^2(s)$, $\ex X_0^2(s)$ and $\ex X_n(s)$ can be computed by explicit formulas \eqref{eq:CIR-mean} and \eqref{eq:CIR-var}.
Thus taking supremum over all $t\in[0,T]$ and denoting
\[
F_n(T) = \int_0^T \left(\ex X_n^2(s) + \sqrt{\ex X_0^2(s) \ex X_n^2(s)}\right) ds, 
\]
we arrive at
\begin{align}
\MoveEqLeft[1]
\sup_{t\in[0,T]} \ex (X_n(t) - X_0(t))^2 
\notag\\
&\le \left(2 \abs{a_n - a_0} + \sigma_0^2 + 2\sigma_0 \abs{\sigma_n -\sigma_0}\right) \left( \abs{a_n - a_0} T + \abs{b_n - b_0} A_0^2(T) + \abs{\sigma_n - \sigma_0} A_0(T) \right)  T e^{b_n T}
\notag\\
& + 2 \abs{b_n - b_0} F_n(T) + \abs{\sigma_n -\sigma_0}^2 A_n^2(T)
+ b_0 \int_0^T \sup_{u\in[0,s]} \ex (X_n(u) - X_0(u))^2\, ds.
\label{exsqdiff2}
\end{align}
Finally, the Gr\"onwall inequality yields
\begin{align*}
\MoveEqLeft[.2]
\sup_{t\in[0,T]} \ex (X_n(t) - X_0(t))^2 
\\
&\le e^{(b_n+b_0)T} \left(2 \abs{a_n - a_0} + \sigma_0^2 + 2\sigma_0 \abs{\sigma_n -\sigma_0}\right) \left( \abs{a_n - a_0} T + \abs{b_n - b_0} A_0^2(T) + \abs{\sigma_n - \sigma_0} A_0(T) \right)  T 
\notag\\
&\quad + e^{b_0 T} \left(2 \abs{b_n - b_0} F_n(T) + \abs{\sigma_n -\sigma_0}^2 A_n^2(T)\right).
\end{align*}
It remains to note that \eqref{eq:CIR-var} implies
\[
\sup_{t\in[0,T]}\ex X_n(s)^2 \le \frac{x_0 (\sigma_n^2 + 2 a_n)}{b_n} \left(1 - e^{-bT}\right) 
+ \frac{a(\sigma_n^2 + 2a_n)}{2b_n^2}\left(1 - e^{-b T}\right)^2
+ x_0^2 = D_n^2(T),
\]
consequently,
\[
F_n(T) \le T\bigl(D_n^2(T) + D_n(T)D_0(T)\bigr).
\]

The proof of claim $(ii)$ follows in a manner similar to that of claim $(i)$. However, in this case, the last term in \eqref{exsqdiff2} vanishes, which simplifies the final steps of the proof, as the application of the Gr\"onwall inequality is no longer required.
\end{proof}

\begin{remark}
Let us now discuss and compare the upper bounds established in Theorems \ref{rateconv}, \ref{rateconv2}, and \ref{rateconv3}.

1. A key advantage of all three theorems is that they provide explicit rates of convergence in terms of the coefficients of the corresponding equations. This makes them particularly valuable for practical analysis.

2. Theorems \ref{rateconv2} and \ref{rateconv3} present upper bounds for the second moments, which are often crucial for practical applications. These bounds cannot be directly obtained from the results for the first moments, such as those provided by Theorem \ref{rateconv}. Furthermore, it is worth noting that, in a similar manner, one can derive upper bounds for $\sup_{t\in[0,T]} \ex (X_n(t) - X_0(t))^{2p}$ for any $p\ge 1$.

3. For a fixed $T > 0$,  the convergence rates established in Theorems \ref{rateconv}, \ref{rateconv2}, and \ref{rateconv3} can be compared as follows. Assume that
\[
\abs{a_n - a_0} \le \delta_n, \quad \abs{b_n - b_0} \le \delta_n, 
\quad \abs{\sigma_n - \sigma_0} \le \delta_n
\]
for some sequence $\delta_n\downarrow 0$ as $n\to\infty$.
Then for the quantity $\sup_{t\in[0,T]} \ex \abs{X_n(t) - X_0(t)} $ Theorems~\ref{rateconv}, \ref{rateconv2}, and \ref{rateconv3} yield rates of convergence of orders $O(\delta_n)$, $O(\delta_n^{1/4})$, and $O(\delta_n^{1/2})$, respectively. Hence, from the perspective of convergence rates, Theorem \ref{rateconv} offers the fastest rate. Similarly, Theorem \ref{rateconv3} demonstrates a superior rate of convergence compared to Theorem \ref{rateconv2}.

4. We can also compare Theorems \ref{rateconv2} and \ref{rateconv3} in the asymptotic case as $T\to\infty$. Note that the functions $R_n(T)$ are bounded, while $A_0^2(T)$ grows linearly with $T$.
Consequently, the right-hand side of \eqref{CIRbound2} behaves as 
$O(T^{1/2} e^{b_n T/2})$, as $T\to\infty$. In contrast, the right-hand side of \eqref{CIRbound3} grows significantly faster, at a rate of $O(T^2 e^{(b_n + b_0) T})$.
From this comparison, for large~$T$, Theorem \ref{rateconv2} provides a tighter bound than Theorem \ref{rateconv3}.
\end{remark}

\section{Parameter estimation}
\label{sec:statistics}
We now address the problem of identifying the squared Bessel process, i.e., the estimation of its parameters. Suppose we have continuous-time observations of a trajectory $\{Y_t, t \in [0, T]\}$ of the squared Bessel process \eqref{BES} over some interval $[0, T]$. Note that these parameters are also defining for the CIR process, and the corresponding estimates can be easily modified for it. Moreover, we can assume that the parameter $\sigma$ is known and focus on estimating the parameter $a$. For continuous-time observations, this assumption is natural, because  $\sigma$ can be determined almost surely from the observations on any fixed interval, as explained in the following remark.

\begin{remark}(Estimation of $\sigma$)
 Let $T > 0$ be fixed, $\delta = \frac{T}{n}$, and $t_k = k\delta$, for $0 \leq k \leq n$. Then
\[
\sum_{k=1}^n \left(Y_{t_k} - Y_{t_{k-1}}\right)^2
\to \sigma^2 \int_0^T Y_s\,ds \quad\text{a.s., as } n\to\infty.
\]
Indeed, as it is clear, quadratic  variation of the linear function tends to zero with the diameter of the partition of the interval.  
This implies that the parameter $\sigma$ can be evaluated using the following identity:
\[
\sigma^2 = \lim_{n\to\infty} \frac{\sum_{k=1}^n \left(Y_{t_k} - Y_{t_{k-1}}\right)^2}{\int_0^T Y_s\,ds}
\quad\text{a.s.}
\]
\end{remark}

To estimate the parameter $a$, we apply the maximum likelihood method. 
First, we transform equation \eqref{BES} using the It\^o formula as follows: 
\begin{align*}
\sqrt{Y_t} &= \sqrt{y_0} + \frac12 \int_0^t \frac{1}{\sqrt{Y_s}}\,dY_s
- \frac18 \sigma^2 \int_0^t \frac{1}{\sqrt{Y_s}}\,ds
\\
 &= \sqrt{y_0} + \frac{a}{2} \int_0^t \frac{ds}{\sqrt{Y_s}} + \frac\sigma2 W_t
- \frac18 \sigma^2 \int_0^t \frac{1}{\sqrt{Y_s}}\,ds
\\
 &= \sqrt{y_0} + \frac{1}{2} \int_0^t \frac{a - \frac{\sigma^2}{4}}{\sqrt{Y_s}}\,ds + \frac\sigma2 W_t.
\end{align*}

The likelihood function is then given by
\[
\frac{d\Qpr}{d\pr}\bigg|_{[0,T]}
= \exp\set{\frac{\frac{\sigma^2}{4} - a}{\sigma} \int_0^T \frac{1}{\sqrt{Y_s}}\,dW_s
- \frac12\left(\frac{\frac{\sigma^2}{4} - a}{\sigma}\right)^2 \int_0^T \frac{ds}{Y_s}},
\]
We aim to find the value of $a$ that maximizes this likelihood function.

To proceed, we set
\[
-\theta = \frac{\sigma}{4} - \frac{a}{\sigma}
\]
and minimize the likelihood function with respect to $\theta$, replacing 
\[
-\theta \int_0^T \frac{1}{\sqrt{Y_s}}\,dW_s
= -\frac{2\theta}{\sigma} \int_0^T \frac{1}{\sqrt{Y_s}}\,d\sqrt{Y_s}
+ \theta^2 \int_0^T \frac{ds}{Y_s}.
\]
Letting $Z_t = \sqrt{Y_t}$, the expression simplifies to minimizing the following:
\[
- \frac{2\theta}{\sigma} \int_0^T \frac{dZ_s}{Z_s} + \frac12 \theta^2 \int_0^T \frac{ds}{Z_s^2}.
\]
This minimization leads to the following maximum likelihood estimator:
\[
\hat\theta_T = \frac{2\int_0^T \frac{dZ_s}{Z_s}}{\sigma \int_0^T \frac{ds}{Z_s^2}}.
\]

\begin{proposition}
Let $2a>\sigma^2$. Then $\hat\theta_T$ is a strongly consistent estimator of $\theta$, i.e.,
\[
\hat\theta_T \to \theta \quad\text{a.s., when } T\to \infty.
\]
\end{proposition}

\begin{proof}
Using the equation
$dZ_t = \frac{\sigma \theta dt}{2 Z_t} + \frac{\sigma}{2}dW_t$,
we represent the estimator $\hat\theta_T$ in the following form
\[
\hat\theta_T 
= \theta + \frac{\int_0^T \frac{dW_s}{Z_s}}{\int_0^T \frac{ds}{Z_s^2}}
= \theta + \frac{M_T}{\langle M \rangle_T},
\]
where $M_T = \int_0^T \frac{dW_s}{Z_s}$ is a locally square-integrable martingale with the quadratic variation
$\langle M \rangle_T = \int_0^T \frac{ds}{Z_s^2}$.
According to the strong law of large numbers for martingales \cite[Ch.~2, \S 6, Thm. 10, Cor. 1]{LipShir}, if $\langle M \rangle_T \to \infty$ a.s., as $T\to\infty$, then  $\frac{M_T}{\langle M \rangle_T}\to 0$ a.s., as $T\to\infty$.
Thus, to establish strong consistency, we need to show that 
\begin{equation}\label{denom-inf}
\int_0^T \frac{ds}{Z_s^2} \to \infty 
\quad \text{a.s., as } T \to \infty.
\end{equation}

Return to the equation
\[
Z_t^2 = Z_0^2 + a t + \sigma \int_0^t Z_s\,dW_s.
\]
We divide both sides by
$\int_0^t Z_s^2\,ds$:
\begin{equation}\label{(1)}
\frac{Z_t^2}{\int_0^t Z_s^2\,ds}
= \frac{Z_0^2}{\int_0^t Z_s^2\,ds} + \frac{a t}{\int_0^t Z_s^2\,ds}
+ \sigma \frac{\int_0^t Z_s\,dW_s}{\int_0^t Z_s^2\,ds}.
\end{equation}

Let $b_n \downarrow 0$, as $n\to\infty$, and consider the approximating CIR processes $\{Z^2_{t,b_n}, t\ge0\}$, $n\ge1$, which solve the equations
\[
Z^2_{t,b_n} = y_0 + \int_0^t \left(a - b_n Z^2_{s,b_n}\right)\,ds + \sigma \int_0^t Z_{s,b_n}\,dW_s,
\quad n \geq 1.
\]
It follows from \eqref{ergo-2} that for any $n > 1$ and any $\varepsilon > 0$ there exists $T_n > 0$ such that for all $T \ge T_n$,
\[
\abs{\frac1T \int_0^T \frac{ds}{Z^2_{s,b_n}} - \frac{b_n}{a - \frac{\sigma^2}{2}}} < \frac{\varepsilon}{2}
\quad\text{a.s.}
\]
Choose $n_0 > 0$ such that for $n \ge n_0$ 
\[
\frac{b_n}{a - \frac{\sigma^2}{2}} < \frac{\varepsilon}{2}.
\]
Then for all $T \ge T_{n_0}$
\begin{equation}\label{CIR-int-inv}
\frac1T \int_0^T \frac{ds}{Z^2_{s,b_{n_0}}} < \varepsilon
\quad\text{a.s.}
\end{equation}
By the comparison theorem \cite[Proposition 2.18, p.~293]{KaratzasShreve}, we have
\[
\pr\left(Z_{t,b_{n_0}}^2 \le Z_t^2 \text{ for all } t\ge0\right) = 1.
\]
Hence, \eqref{CIR-int-inv} implies that for all $T \ge T_{n_0}$
\[
\frac1T \int_0^T \frac{ds}{Z^2_s} < \varepsilon
\quad\text{a.s.}
\]
Consequently,
\begin{equation}\label{BES-int-inv}
\frac1T \int_0^T \frac{ds}{Z^2_s} \to 0, \quad
\text{a.s., as } T \to \infty.
\end{equation}
This implies that
\[
\frac1T \int_0^T Z^2_s\,ds \to \infty,
\quad\text{ a.s., as } t\to\infty,
\]
since
\[
\frac{1}{T^2} \int_0^T Z^2_s\,ds \int_0^T \frac{ds}{Z^2_s} \ge 1
\]
by the Cauchy--Schwarz inequality.

Thus, we obtain
\[
\frac{t}{\int_0^t Z_s^2\,ds} \to 0 
\quad \text{and} \quad
\frac{\int_0^t Z_s\,dW_s}{\int_0^t Z_s^2\,ds}  \to 0
\quad\text{a.s., as } t \to \infty.
\]
From \eqref{(1)}, it follows that 
\begin{gather*}
\frac{Z_t^2}{\int_0^t Z_s^2\,ds} \to 0 \quad \text{a.s., as } t\to\infty,
\shortintertext{or}
\frac{\int_0^t Z_s^2\,ds}{Z_t^2} \to \infty \quad \text{a.s., as } t\to\infty.
\end{gather*}

Therefore, for almost all $\omega$ and for all $C > 0$ there exists $t_0$ such that for all $t \ge t_0$
\[
\frac{1}{Z_t^2} \ge \frac{C}{\int_0^t Z_s^2\,ds},
\]
which implies that
\[
\int_{t_0}^\infty \frac{ds}{Z_s^2} \ge C \int_{t_0}^\infty \frac{ds}{\int_0^s Z_u^2\,du},
\]
Letting $C \to \infty$, we obtain
\[
\int_{t_0}^\infty \frac{ds}{Z_s^2} = \infty
\quad\text{a.s.}
\]
Thus, condition \eqref{denom-inf} is satisfied, and we conclude that strong consistency holds.
\end{proof}

\begin{corollary}
Let $2a>\sigma^2$.
The maximum likelihood estimator of the parameter $a$ of the squared Bessel process \eqref{BES} is given by
\[
\hat a_T = \sigma \hat\theta_T + \frac{\sigma^2}{4}
= \frac{2\int_0^T \frac{d\sqrt{Y_s}}{\sqrt{Y_s}}}{\int_0^T \frac{ds}{Y_s}}
+ \frac{\sigma^2}{4}
\]
and it is strongly consistent, as $T\to\infty$.
\end{corollary}

\section{Instability and some functional limit theorems for the squared Bessel process}
\label{sec:instability}
As we have already mentioned, the process $X$ is ergodic, while $Y$ is non-ergodic.
Now our goal is to establish how these properties reflect in the notion of stochastic instability. There are several approaches to stochastic instability of the processes. We shall consider the following definition, see book \cite{KKM}.
\begin{definition} A stochastic process $\xi$ is called stochastically unstable if for any constant $N>0$
\[
\lim_{t\to +\infty} \frac{1}{t} \int_0^t \pr\set{|\xi_s| < N} ds
= 0.
\]
    \end{definition}
\begin{proposition}\label{Prop.5.1}
Let $N > 0$ be an arbitrary constant.
\begin{itemize}
\item[$(i)$]
The CIR process $X$ has the following property:
\[
\lim_{t\to +\infty} \frac{1}{t} \int_0^t \pr\set{X_s < N} ds
= \frac{\gamma\left(\frac{2a}{\sigma^2}, \frac{2bN}{\sigma^2}\right)}{\Gamma\left(\frac{2a}{\sigma^2}\right)},
\]
where
$\gamma(a,x) = \int_0^x u^{a-1} e^{-u}\,du$
is the lower incomplete Gamma function.

\item[$(ii)$]
The squared Bessel process $Y$ is stochastically unstable, i.e.,
\[
\lim_{t\to +\infty} \frac{1}{t} \int_0^t \pr\set{Y_s < N} ds = 0.
\]
\end{itemize}
\end{proposition}
\begin{proof}
$(i)$
Using the convergence \eqref{eq:gammadens}, we obtain
\begin{align*}
\pr\set{X_t < N} &= \int_0^N p_t(x)\,dx
\to \int_0^N p_\infty(x)\,dx
= \frac{(2b/\sigma^2)^{2a/\sigma^2}}{\Gamma(2a/\sigma^2)} \int_0^N x^{2a/\sigma^2 - 1}e^{- 2bx/\sigma^2}\,dx
\\
&= \frac{1}{\Gamma(2a/\sigma^2)} \int_0^{2bN/\sigma^2} u^{2a/\sigma^2 - 1}e^{-u}\,du
= \frac{\gamma\left(\frac{2a}{\sigma^2}, \frac{2bN}{\sigma^2}\right)}{\Gamma\left(\frac{2a}{\sigma^2}\right)},
\quad t \to +\infty.
\end{align*}
Taking into account continuity of the integrand and applying the l'H\^opital's rule, we get
\[
\lim_{t\to +\infty} \frac{1}{t} \int_0^t \pr\set{X_s < N} ds
= \lim_{t\to +\infty} \pr\set{X_t < N}
= \frac{\gamma\left(\frac{2a}{\sigma^2}, \frac{2bN}{\sigma^2}\right)}{\Gamma\left(\frac{2a}{\sigma^2}\right)}.
\]

$(ii)$ The proof is similar to that of the statement $(i)$. However, we get zero limit, since the probability density $g_t(x)$ of the process $Y$ converges to zero, as $t\to\infty$, according to \eqref{denstozero}.
\end{proof}
\begin{remark} Proposition \ref{Prop.5.1} justifies why we say that  we have phase transition when coefficient $b$ of CIR process tends to zero and we get squared Bessel process. Indeed, squared Bessel process is stochastically unstable while, oppositely,  CIR process is ergodic, and in this sense, stochastically stable.

\end{remark}
Now, note that in some sense, we  were lucky because knowledge of the distribution allowed  us to establish instability of the squared Bessel process directly, without application of the tools of stochastic analysis. However, we can establish instability for the approximations of the  squared Bessel process whose distribution is unknown.
So, again, consider  the squared Bessel process determined as the unique solution of the equation \eqref{BES}.
Let   for simplicity   $\sigma = 2$. General condition $a \ge \frac{\sigma^2}{2}$ leads in our case to $a \ge 2$. So, we assume that $a\ge 2$, then
  the trajectories of the solution are strictly positive with probability 1. Therefore,  we can consider function $F(x) = \sqrt{x}$ and apply It\^o formula to $Y_t$, obtaining equation
\begin{align*}
\sqrt{Y_t} & = \sqrt{y_0} + \frac12 \int_0^t \frac{a - \frac{\sigma^2}{4}}{\sqrt{Y_s}}\, ds + \frac{\sigma}{2} W_t
\\
 & = \sqrt{y_0} + \frac12 \int_0^t \frac{a - 1}{\sqrt{Y_s}}\, ds + W_t,
\end{align*}
or, that is the same,
\[
V_t = V_0 + \frac12 \int_0^t \frac{a - 1}{V_s}\, ds + W_t,
\]
where $V_t = \sqrt{Y_t}$, $V_0 = \sqrt{Y_0}$.
Note that $a - 1 \ge 1$.
Now our goal is to consider a smooth version of Bessel process, i.e., a solution of the stochastic differential equation
\begin{equation}\label{(vvv)}
dV_t^\varepsilon = \frac{c dt}{\sqrt{(V_t^\varepsilon)^2 + \varepsilon^2}} + dW_t,
\end{equation}
where
$\varepsilon \ne 0$, $c > 0$, $V_0^\varepsilon = \sqrt{y}_0 > 0$.
The coefficient
$\frac{c}{\sqrt{x^2 + \varepsilon^2}}$
is Lipschitz because it has a bounded derivative:
\[
\abs{\left( \frac{c}{\sqrt{x^2 + \varepsilon^2}}\right)'}
= \frac{c \abs{x}}{\left(x^2 + \varepsilon^2\right)^{3/2}}
\le c \,  \frac{\abs{x}}{\left(x^2 + \varepsilon^2\right)^{1/2}} \cdot \frac{1}{x^2 + \varepsilon^2}
\le \frac{c}{\varepsilon^2}.
\]
Also it is bounded, therefore, due the standard existence-uniqueness theorem for stochastic differential equations,  equation \eqref{(vvv)} has a unique strong solution.

We want to achieve three goals.
First, we establish the convergence of $V^\varepsilon$ to the (non-squared) Bessel process $V=\sqrt{Y}$, as $\varepsilon \to 0$, that is more or less the expected result.
\begin{proposition}
  Let $c=\frac12(a-1),$  $\varepsilon \to 0$. Then for any $t > 0$
\[
V_t^\varepsilon \to V_t \quad\text{a.s.}
\]
\end{proposition}
\begin{proof}

 Consider the difference
\[
\abs{V_t^\varepsilon - V_t}
= c \abs{\int_0^t \frac{ds}{\sqrt{(V_s^\varepsilon)^2 + \varepsilon^2}}
- \int_0^t \frac{ds}{V_s}}.
\]
According to the comparison theorem,
$V_t^\varepsilon \le V_t$, $t \ge 0$, $\varepsilon > 0$, a.s.
Therefore
\[
\int_0^t \frac{ds}{\sqrt{(V_s^\varepsilon)^2 + \varepsilon^2}}
\le \int_0^t \frac{ds}{V_s},
\quad t \ge 0, \; \varepsilon > 0, \text{ a.s.}
\]
Moreover, $Y^\varepsilon_t$ increases in $\varepsilon$ when $\varepsilon$ decreases, therefore
$\int_0^t \frac{ds}{\sqrt{(V_s^\varepsilon)^2 + \varepsilon^2}}$
also increases and has a limit.

Since
\[
\int_0^t \frac{ds}{\sqrt{(V_s^\varepsilon)^2 + \varepsilon^2}}
\ge \int_0^t \frac{ds}{\sqrt{V_s^2 + \varepsilon^2}}
\uparrow \int_0^t \frac{ds}{V_s},
\quad\text{as } \varepsilon\downarrow 0,
\]
we get from the double inequality
\[
\int_0^t \frac{ds}{\sqrt{V_s^2 + \varepsilon^2}}
\le \int_0^t \frac{ds}{\sqrt{(V_s^\varepsilon)^2 + \varepsilon^2}}
\le \int_0^t \frac{ds}{V_s},
\]
that
\[
 \int_0^t \frac{ds}{\sqrt{(V_s^\varepsilon)^2 + \varepsilon^2}}
\to \int_0^t \frac{ds}{V_s},
\quad\text{a.s., as } \varepsilon\downarrow0,
\]
whence
\[
V_t^\varepsilon = V_0 + \int_0^t \frac{c\, ds}{\sqrt{(V_s^\varepsilon)^2 + \varepsilon^2}} + W_t \to V_t,
\quad\text{a.s., as } \varepsilon\downarrow0, \text{ for } t \ge 0.
\]
\end{proof}
Second, we wish to establish stochastic instability of $V^\varepsilon $.
\begin{proposition}\label{Prop. 5.2}
Let $c>0$. Then the processes $V^\varepsilon $ are  stochastically unstable for any $\varepsilon \neq 0$, i.e.,
\[
\lim_{t\to +\infty} \frac{1}{t} \int_0^t \pr\set{V^\varepsilon_s < N} ds = 0
\]
for   any constant $N > 0$.
\end{proposition}
\begin{remark} Stochastic instability of $V^\varepsilon$ does not follow from Proposition \ref{Prop.5.1} because $V^\varepsilon\le V$. Oppositely, stochastic instability of $V$ follows from Proposition \ref{Prop. 5.2}. However, having explicit formulas for the distributions of $X$ and $Y=(V)^2$, we preferred to give the direct proof to Proposition \ref{Prop.5.1}.

\end{remark}
\begin{proof} We shall apply Theorem  3.1, item 2. from \cite{KKM}, see Theorem \ref{th:KKM31} in Appendix.
So, in terms of Section 3.1 from \cite{KKM}, the process $V^\varepsilon$ is the unique solution of the equation
\[
dV^\varepsilon_t = a\left(V^\varepsilon_t\right) dt + dW_t, \quad t > 0, \; V_0 > 0,
\]
with the drift $a(x) = \frac{c}{\sqrt{x^2+ \varepsilon^2}}$. Applying  the l'H\^opital's rule, let  us check the value of the limit
\[\lim_{|x| \to +\infty}\frac{1}{\log|x|}\int_0^x a(v)dv=\lim_{x \to +\infty}
|x|a(x)=c>0,\]
and the proof immediately follows from Theorem \ref{th:KKM31} ($c_{0}=c>0$,\;$2c_{0} >-1$).
\end{proof}

And finally, we establish a bit  unexpected statement: if $\varepsilon \ne 0$ is fixed, then the properly normalised procesess $V^\varepsilon$ weakly converge    to the (non-squared) Bessel process as time tends to infinity. The definition of weak convergence is given in Appendix, Definition \ref{KKMdef2}.

\begin{theorem}
 Normalised stochastic process
$Z_\varepsilon(t) = \frac{Y^\varepsilon_{tT}}{\sqrt{T}}$
converges weakly, as $T \to \infty$, to the Bessel process $Y_t$ that is the solution of the equation
\begin{equation}\label{sdelimproc}
Y_t^2 = 3t + 2\int_0^t Y_s \,dW_s, \quad t \ge 0.
\end{equation}

\end{theorem}

\begin{proof}
 We shall apply Theorem \ref{th:KKM33} from Appendix.
In our case
\[
\frac{1}{x} \int_0^x v a_\varepsilon(v)\,dv
= \frac{1}{x} \int_0^x \frac{v}{\sqrt{v^2 + \varepsilon^2}}\,dv\to
\begin{cases}
1,  & x \to +\infty, \\
-1, & x \to -\infty,
\end{cases}
\]
for any $\varepsilon > 0$.
Therefore we have that $c_{1}=1$,\;$c_{2}=-1$ in Theorem \ref{th:KKM33} and
\[
2 c_1 = 2 > 1 \quad\text{and}\quad 2 c_2 = -2 < 1.
\]
Then the stochastic process
$\frac{Y^\varepsilon_{tT}}{\sqrt{T}}$
weakly converges, as $T\to\infty$, to the Bessel process satisfying equation \eqref{sdelimproc}.
\end{proof}

\section{Appendix}

\subsection{Special functions}\label{app:special}
In this subsection we provide definitions and selected properties of special functions, used in the paper. We refer to the handbooks and \cite[Chapter 13]{Abramowitz} or \cite[Chapter 47]{An-Atlas-of-Functions} for more details.

The confluent hypergeometric function ${}_1F_1$ (Kummer's function)
is defined by the series:
\[
{}_1F_1(a,c,x) = \sum_{j=0}^\infty  \frac{(a)_j}{(c)_j} \frac{x^j}{j!} ,
\]
where $(a)_n = a (a+1)\dots(a+n-1)$, $(a)_0 = 1$.
When $a$ is a negative integer, say $a = -n$, this function becomes a polynomial of order $n$. Specifically, it is expressed via the generalized Laguerre polynomial $L_n^{(c-1)}$:
\[
{}_1F_1(-n,c,x)
= \frac{n!}{(c)_n} L_n^{(c-1)} (x)
=   \sum_{j=0}^n \binom{n}{j} \frac{(-x)^j}{(c)_j}.
\]
The following identity is known as Kummer’s transformation:
\[
{}_1F_1(a,c, -x) = e^{-x}{}_1F_1(c-a,c, x).
\]
Therefore, combining two above formulas, we derive that for $n\in\mathbb N$,
\[
e^{-x} {}_1F_1(c+n,c,x) = {}_1F_1(-n,c,-x)
= \sum_{j=0}^n \binom{n}{j} \frac{x^j}{(c)_j}.
\]
In particular,
\begin{gather}
e^{-x} {}_1F_1(c+1,c,x) = 1 + \frac{x}{c}, \quad
e^{-x} {}_1F_1(c+2,c,x) = 1 + \frac{2x}{c} + \frac{x^2}{c(c+1)},
\notag\\
e^{-x} {}_1F_1(c+3,c,x) = 1 + \frac{3x}{c} + \frac{3x^2}{c(c+1)} + \frac{x^3}{c(c+1)(c+2)}\label{1F_1(c+3)}.
\end{gather}

\subsection{A limit theorem for equations with nonhomogeneous coefficients
and non-Lipschitz diffusion}
\label{app:lim-thm}

Here we present a limit theorem from \cite{MPS09}.
Consider the following sequence of stochastic differential equations:
\begin{equation*}
X_{n}(t)=X_{n}(0)+\int_{0}^{t} b_{n}\bigl(s, X_{n}(s)\bigr) d s+\int_{0}^{t} \sigma_{n}\bigl(s, X_{n}(s)\bigr) d W(s), \quad n \geq 0,
\end{equation*}
where the initial conditions $X_{n}(0)$ are nonrandom. Assume that the coefficients of these equations satisfy the following conditions:
\begin{itemize}
\item[$(\mathrm{Y} 1_{n})$] $b_{n}$ and $\sigma_{n}$ are continuous with respect to all arguments;

\item[$(\mathrm{Y} 2_{n})$] linear growth:
\begin{equation*}
\left|\sigma_{n}(t, x)\right|+\left|b_{n}(t, x)\right| \leq L(1+|x|), \quad t \geq 0,\; x \in \real;
\end{equation*}

\item[$(\mathrm{Y} 3_{n})$] Lipschitz condition for $b_{n}$:
\[
\left|b_{n}(t, x)-b_{n}(t, y)\right| \leq L|x-y|, \quad t \geq 0,\; x, y \in \real;
\]

\item[$(\mathrm{Y} 4_{n})$] there exists an increasing function $\rho_{n}\colon \real^+ \to \real^+$ such that $\int_{0+} \rho_{n}^{-2}(u) d u=\infty$ and
$$
\left|\sigma_{n}(t, x)-\sigma_{n}(t, y)\right| \leq \rho_{n}(|x-y|), \quad t \geq 0,\quad x, y \in \real.
$$
\end{itemize}

Additionally, assume that as $n \to \infty$, the following convergences hold:
\begin{equation}\label{cond-conv}
X_{n}(0) \to X_{0}(0), \quad b_{n}(t, x) \to b_{0}(t, x), \quad \sigma_{n}(t, x) \to \sigma_{0}(t, x)
\end{equation}
for $t \geq 0$ and $x \in \real$.

\begin{theorem}[{\cite[Theorem 4.1]{MPS09}}]\label{th:MPS09}
If conditions $(\mathrm{Y} 1_{n})$--$(\mathrm{Y} 4_{n})$ and \eqref{cond-conv} hold, then
\begin{equation*}
\ex\left[\left|X_{n}(t)-X_{0}(t)\right|\right] \to 0, \quad n \to \infty,
\end{equation*}
uniformly in any finite interval.
\end{theorem}
\subsection{Functional and other limit theorems for the solutions of stochastic differential equations}
Let us consider an equation of the form
\begin{equation}\label{eq3.1}
d\xi_t=a(\xi_t)dt+dW_t,\;\;\; t>0, \;\;\; \xi_{0}=x_{0} ,
\end{equation}
with real measurable drift coefficient satisfying additional assumption: $\left|x\, a(x)\right|\le M$   for a certain constant $M$ and for all $x\in \real$.
We  use the following notation:
\begin{equation*} 
\psi (x,c)=\frac{1}{\log \left|x\right|} \int\limits _{0}^{x}a(v)dv-c.
\end{equation*}

\begin{theorem}[{\cite[Theorem 3.1, item 2.]{KKM}}]\label{th:KKM31}
Let $\xi $ be a solution to equation~(\ref{eq3.1}) and let

\begin{equation*} 
\mathop{\lim }\limits_{\left|x\right|\to +\infty } \psi (x,c_{0} )=0.
\end{equation*}
Then for $2c_{0} >-1$ the solution $\xi $ is stochastically unstable.
\end{theorem}

For investigating the behavior when $T\to +\infty$  of the distribution of the normalized random process $\xi _{T} (t)=\frac{\xi_{tT}}{\sqrt{T}}$, $t>0$, where $T$ is  parameter, we study the weak convergence to some limit  process $ \zeta $ in the following sense.

\begin{definition}\label{KKMdef2}
A family $\xi _{T}=\{\xi _{T} (t),\, t\geq 0\}$  of stochastic processes is said to converge weakly,
as  $T\to +\infty $, to a process $\zeta=\{\zeta (t),\, t\geq 0\}$  if, for any $L>0$, the measures $\mu _{T} [0,L]$,  generated
by the processes  $\xi _{T} (\cdot)$  on the interval $[0,L]$ converge weakly to the measure $\mu [0,L]$ generated by the process $\zeta (\cdot)$.
\end{definition}

\begin{remark} Since the processes $ \xi_{T}  $ are continuous with probability 1 as the solutions to It\^{o}'s stochastic differential equations, Definition~\ref{KKMdef2} is a definition of the weak convergence of the processes $ \xi_{T} $ to the continuous process $ \zeta $ in a uniform topology of the space of continuous functions.
\end{remark}

\begin{theorem}[{\cite[Theorem 3.3(2)]{KKM}}]\label{th:KKM33}
Let $\xi $ be a solution to equation~(\ref{eq3.1}), and let there exist the constants $c_1$ and $c_2$ such that
\begin{equation*} 
\mathop{\lim }\limits_{\left|x\right|\to +\infty } \left[\frac{1}{x} \int\limits _{0}^{x}va(v)dv -\overline{c}(x)\right]=0,\quad     \overline{c}(x)=\left\{\begin{array}{l} {c_{1} ,\, \, \, x\geq 0,} \\ {c_{2} ,\, \, \, x<0.} \end{array}\right.
\end{equation*}

 If $2c_{1} >1$ and $2c_{2} <1,$ then the stochastic process
$\left|\xi_{tT}\right|T^{-\frac{1}{2} } $ converges weakly, as $T\to +\infty $, to the process $Y$, which is the solution of It\^{o}'s stochastic differential equation
\begin{equation*} 
Y_t^{2} =(2c+1)t+2\int\limits _{0}^{t}Y_s\,dW_s
\end{equation*}
for $c=c_{1}.$
\end{theorem}

\section*{Acknowledgments}
The first author is supported by The Swedish Foundation for Strategic Research, grant Nr. UKR24-0004, by Japan Science
and Technology Agency CREST, project reference number 811JPMJCR2115 and the ToppForsk project no. 274410 of the Research
Council of Norway with the title STORM: Stochastics for Time-Space Risk Models.
The second author is supported by the Research Council of Finland, decision number 359815.

\nocite*
\bibliographystyle{abbrv}
\bibliography{bessel}
\end{document}